\title{A quasi-isometry invariant and thickness bounds for right-angled Coxeter groups}
\author{Ivan Levcovitz}
\date{}

\documentclass[12pt]{article}
\usepackage{amsmath}
\usepackage{amssymb}
\usepackage{amsthm}

\usepackage{hyperref}

\usepackage[normalem]{ulem}
\usepackage{graphicx}
\usepackage{overpic}

\newtheorem{theorem}{Theorem}[section]
\newtheorem{lemma}[theorem]{Lemma}

\newtheorem{corollary}{Corollary}[theorem]

\newtheorem{conjecture}[theorem]{Conjecture}

\newtheorem{introtheorem}{Theorem}

\theoremstyle{definition}
\newtheorem{definition}[theorem]{Definition}

\theoremstyle{remark}
\newtheorem{remark}{Remark}[theorem]

\usepackage[bottom=1in, left=1in, right=1in, top=1in]{geometry}

\usepackage{sectsty}

\sectionfont{\fontsize{12}{15}\selectfont \centering}

\subsectionfont{\fontsize{12}{15}\selectfont }

\begin{document}
\maketitle

\begin{abstract}
We introduce a new quasi-isometry invariant of $2$--dimensional right-angled Coxeter groups, the hypergraph index, that partitions these groups into infinitely many quasi-isometry classes, each containing infinitely many groups. Furthermore, the hypergraph index of any right-angled Coxeter group can be directly computed from the group's defining graph. The hypergraph index yields an upper bound for a right-angled Coxeter group's order of thickness, order of algebraic thickness and divergence function. Finally, given an integer $n>1$, we give examples of right-angled Coxeter groups which are thick of order $n$, yet are algebraically thick of order strictly larger than $n$, answering a question of Behrstock-Dru\c{t}u-Mosher.
\end{abstract}

\section{Introduction}

Recently there has been considerable interest in the quasi-isometric classification of right-angled Coxeter groups \cite{ChSu} \cite{BHS_qflats} \cite{BHSC} \cite{BFHS} \cite{DST} \cite{DT} \cite{DT2} \cite{Lev}. As every right-angled Artin group is finite index in some right-angled Coxeter group \cite{DJ}, results on the quasi-isometric classification of right-angled Artin groups are steps in the classification of right-angled Coxeter groups as well \cite{BC} \cite{BKS} \cite{BN} \cite{BN2} \cite{Huang} \cite{Huang2} \cite{BH}.

The right-angled Coxeter group, $W_{\Gamma}$, has a presentation consisting of an order $2$ generator for each vertex of a simplicial graph $\Gamma$ with the relation that two generators commute if there is an edge between the corresponding vertices of $\Gamma$. We introduce the hypergraph index, which takes the value of a non-negative integer or infinity, of a right-angled Coxeter group. The group $W_{\Gamma}$ is called $2$--dimensional if $\Gamma$ does not contain a $3$--cycle. We show the hypergraph index gives a decomposition of $2$--dimensional right-angled Coxeter groups into distinct quasi-isometry classes: 

\begin{introtheorem} \label{intro_thm_qi}
	The hypergraph index is a quasi-isometry invariant of $2$--dimensional right-angled Coxeter groups.
\end{introtheorem}

A right-angled Coxeter group has infinite hypergraph index if and only if it is relatively hyperbolic. However, in this case the spectrum of hypergraph indexes of the maximal peripheral subgroups provides a more refined quasi-isometry invariant (see Corollary \ref{cor_hypergraph_index_spectrum}).

\begin{figure}[htp]
	\centering
	\begin{overpic}[scale=.6]{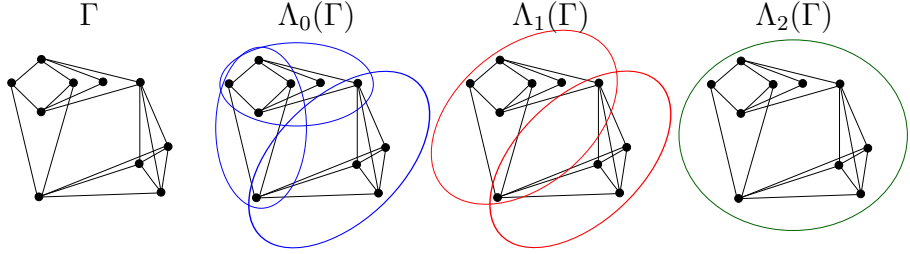}
		\put(8,25){$\Gamma$}		
		\put(30,25){$\Lambda_0(\Gamma)$}
		\put(56,25){$\Lambda_1(\Gamma)$}
		\put(83,25){$\Lambda_2(\Gamma)$}
		
	\end{overpic}
	
	\caption{The hypergraphs $\{\Lambda_i(\Gamma)\}$ associated to the graph $\Gamma$. The hypergraph $\Lambda_0(\Gamma)$ has two hyperedges corresponding to wide subgraphs and several strip subgraph hyperedges  (one is shown). As a hyperedge of $\Lambda_2(\Gamma)$ contains every vertex, the right-angled Coxeter group $W_{\Gamma}$ has hypergraph index $2$. For the relevant definitions, see definition \ref{def_lambda_hypergraphs} and \ref{def_hypergraph_index}.} \label{fig_hypergraph_example}
\end{figure}

The hypergraph index of a right-angled Coxeter group is obtained through an easily computable graph-theoretic construction on the group's defining graph, $\Gamma$, that outputs a sequence of hypergraphs, $\{\Lambda_i\}$, each having the same vertex set as $\Gamma$. The hyperedges of the first hypergraph, $\Lambda_0$, are subgraphs of $\Gamma$ which are certain types of graph joins. The hyperedges of $\Lambda_{i+1}$ are certain unions of hyperedges of $\Lambda_i$. The hypergraph index of $W_{\Gamma}$ is the smallest integer $h$ such that some hyperedge of $\Lambda_h$ contains every vertex (see Figure \ref{fig_hypergraph_example} for an example). If no such $h$ exists, then we set $h = \infty$. The computation time of the hypergraph index is always bounded by the number of vertices of $\Gamma$, even when the hypergraph index happens to be infinite.
	
Other invariants have been used to study the quasi-isometric classification of right-angled Coxeter groups. The Bowditch boundary has been recently used to understand certain classes of $2$--dimensional hyperbolic right-angled Coxeter groups \cite{DT2} \cite{DST}. The authors of \cite{ChSu} also use their notion of a contracting boundary to differentiate between certain relatively hyperbolic right-angled Coxeter groups.
	
When the tools of relative hyperbolicity are not available, there are still a few quasi-isometry invariants to distinguish between these groups. Two such invariants are the order of thickness of a group, and the group's divergence function, a measure of how quickly geodesic rays in the group's Cayley graph can stray apart.
	
The authors of \cite{BHSC} show every right-angled Coxeter group that is not relatively hyperbolic must be thick of some order. On the other hand, the authors of \cite{DT} and \cite{Lev}, provide criteria to determine the divergence of certain right-angled Coxeter groups. However, despite a few exceptions, the exact order of thickness and the divergence function of most right-angled Coxeter groups is unknown. An advantage of the hypergraph index over these invariants is that it is always computable. Furthermore, we show the hypergraph index is strongly related with the order of thickness and the divergence of these groups: 
	
\begin{introtheorem} \label{intro_thm_bounds}
	Suppose the right-angled Coxeter group, $W_{\Gamma}$, has hypergraph index $h \neq \infty$, then $W_{\Gamma}$ is thick of order at most $h$ and the divergence of $W_{\Gamma}$ is bounded above by a polynomial of degree $h+1$.
\end{introtheorem}
	
The hypergraph index also provides an upper bound on the order of \textit{algebraic} thickness of a right-angled Coxeter group (see Theorem \ref{thm_alg_thick_upper_bound}). 

For both $n=0$ and $n=1$, in the class of right-angled Coxeter groups, thickness of order $n$, algebraic thickness of order $n$, polynomial divergence of degree $n+1$ and hypergraph index $n$ are all equivalent notions (see section \ref{sec_thick_bounds} for an overview). Actually, the following conjecture seems to hold for all groups whose divergence and thickness we can compute:

\begin{conjecture}
	Let $\Gamma$ be a simplicial graph and $W_{\Gamma}$ the corresponding right-angled Coxeter group. The following are equivalent: 
	\begin{enumerate}
		\item $\Gamma$ has hypergraph index $n$. 
		\item $W_{\Gamma}$ is thick of order $n$.
		\item The divergence of $W_{\Gamma}$ is a polynomial of degree $n+1$. 
	\end{enumerate}
\end{conjecture}

One may then ask if algebraic thickness of order $n$ and thickness of order $n$ are equivalent notions in right-angled Coxeter groups, as this is true for $n=0$ and $n=1$. In fact, in the paper where thick groups are originally defined, the authors ask if the order of algebraic thickness of any finitely generated group is equivalent to the group's order of thickness \cite[Question 7.7]{BDM}. Sisto provided a negative answer to this question by demonstrating an example of a group which is thick of order $1$ but is not algebraically thick of order $1$  \cite{BD}. We give a negative answer to this question for the case of higher orders of thickness (see Theorem \ref{thm_counterexample} for a more detailed statement): 
	
\begin{introtheorem} \label{intro_thm_counterexample}
	Given any integer $n>1$, there are right-angled Coxeter groups which are thick of order $n$, but are algebraically thick of order strictly larger than $n$.
\end{introtheorem}

The paper is organized as follows. Section \ref{sec_background} reviews some of the necessary background. We define the lambda hypergraphs and the hypergraph index associated to a right-angled Coxeter group in section \ref{sec_lambda_graphs} where we also prove some essential results regarding these constructions. We define the notion of coarse intersection degree, an important notion for many of the results in this paper in section \ref{sec_coarse_intersection}. In section \ref{sec_qi} we prove Theorem \ref{intro_thm_qi}. In section \ref{sec_thick_bounds}, we review how the hypergraph index relates to some known classes of right-angled Coxeter groups, and we prove Theorem \ref{intro_thm_bounds}. Finally, section \ref{sec_thick_not_alg} is devoted to proving Theorem \ref{intro_thm_counterexample}.

\vspace{.4cm}

\noindent \textbf{Acknowledgements:} 
I would like to thank my advisor, Jason Behrstock, for his extremely helpful guidance and excellent suggestions. I also thank Ruth Charney, Mike Davis, Pallavi Dani, Mark Hagen, Jean-Fran\c{c}ois Lafont, Emily Stark, Tim Susse, Anne Thomas and Hung Cong Tran for helpful discussions regarding the quasi-isometric classification of right-angled Coxeter groups. I am also grateful to the anonymous referee for the many helpful suggestions and corrections.

\section{Background} \label{sec_background}
Throughout the text, given a metric space $X$ and $Z \subset X$ a subspace, we denote the $C$--neighborhood of $Z$ by $N_C(Z)$.

Let $X$ and $Y$ be metric spaces. A \textit{($K$, $C$)--quasi-isometry} is a (not necessarily continuous) function $f: X \to Y$, such that for all $a, b \in X$ we have:
\[ \frac{1}{K}d_X(a,b) - C \le d_Y(f(a), f(b)) \le Kd_X(a,b) + C \]
Furthermore, we require $f$ to be coarsely surjective so that  $Y = N_C(f(X))$. Quasi-isometries provide a natural notion of equivalence in a coarse geometric setting. For a detailed background on quasi-isometries and geometric group theory in general, see \cite{BH}.

\subsection{Graphs and hypergraphs}
 We summarize the common graph constructions we use throughout the paper. $\Gamma$ will always denote a simplicial graph. $V(\Gamma)$ and $E(\Gamma)$ are respectively the vertex set and edge set of $\Gamma$.

A \textit{clique} is a graph with the property that any two vertices are adjacent.  A $k$--clique is a clique with $k$ vertices. $\Gamma$ is a \textit{join} if there are subgraphs $\Gamma_1$, $\Gamma_2 \subset \Gamma$ such that $V(\Gamma) = V(\Gamma_1) \cup V(\Gamma_2)$, and every vertex in $\Gamma_1$ is adjacent to every vertex in $\Gamma_2$. Graph joins are written as: $\Gamma = \Gamma_1 \star \Gamma_2$.

Given a graph $\Gamma$ and a subset of its vertices, $T \subset V(\Gamma)$, the subgraph \textit{induced} by $T$ is the graph with vertex set $T$ and with the property that two vertices are adjacent if and only if they are adjacent in $\Gamma$.

Given a vertex $v \in \Gamma$, the \textit{link of $v$} is the set, $Link(v) = \{s \in V(\Gamma) | (v,s) \in E(\Gamma) \}$. The \textit{star of $v$} is the set, $Star(v) = Link(v) \cup v$.

A \textit{hypergraph}, $\Lambda$, consists of a set of vertices, $V(\Lambda)$, and a set of hyperedges, $\mathcal{E}(\Lambda)$. A hyperedge, $E \in \mathcal{E}(H)$, is a subset of $V(H)$ consisting of any number of vertices. Note that a graph is a hypergraph whose hyperedges each contain two vertices.

\subsection{Right-angled Coxeter groups}

Given a simplicial graph $\Gamma$ with vertex set $S = \{s_1, s_2, ..., s_n\}$ and edge set $E$, the corresponding right-angled Coxeter group is given by the presentation:
	\[W_{\Gamma} = \langle S ~| ~ s_i^2 = 1, s_is_j = s_js_i \text{ for } (s_i, s_j) \in E  \rangle \]
	
We refer the reader to \cite{BB} and \cite{Dav} for a nice background on Coxeter groups. All results from this subsection are proved in these references.

\begin{definition} [Special subgroup] \label{induced_subgroup_def}
	Let $W_{\Gamma}$ be a right-angled Coxeter group with generating set $S = V(\Gamma)$. For $T \subset S$, let $W_T$ be the subgroup of $W_{\Gamma}$ generated by the induced subgraph of $T$. $W_T$ is called a \textit{special subgroup}.
\end{definition}

The notation in the above definition is justified by the following result:

\begin{lemma} \label{lemma_special_subgroups}
	Let $W_{\Gamma}$ be a right-angled Coxeter group with generating set $S = V(\Gamma)$. Given $T \subset S$, let $\Delta \subset \Gamma$ be the subgraph induced by $T$. $W_{T}$ is isomorphic to $W_{\Delta}$. Furthermore, $W_{T}$ is convex with respect to the word metric of $W_{\Gamma}$.
\end{lemma}

It is readily checked that $W_{\Gamma}$, is finite if and only if $\Gamma$ is a clique. Furthermore, given induced subgraphs $\Gamma_1$, $\Gamma_2 \subset \Gamma$, $W_{\Gamma} = W_{\Gamma_1} \times W_{\Gamma_2}$ if and only if $\Gamma = \Gamma_1 \star \Gamma_2$   

\subsection{The Davis complex, $\Sigma_{\Gamma}$, of a right-angled Coxeter group}

Given a right-angled Coxeter group, $W_{\Gamma}$, we describe its \textit{Davis complex}, $\Sigma_{\Gamma}$, a natural CAT(0) cube complex on which $W_{\Gamma}$ acts geometrically. We again refer the reader to $\cite{{Dav}}$ for a detailed background on the Davis complex. 

The $1$--skeleton of $\Sigma_{\Gamma}$ is the Cayley graph of $W_{\Gamma}$ where edges are given unit length. For every $k$--clique, $T \subset \Gamma$, the subgroup $W_T$ is isomorphic to the product of $k$ copies of $\mathbb{Z}_2$. It follows that the Cayley graph of $W_T$ is isometric to a unit $k$--cube. For each coset, $gW_T$, where $T$ is a $k$--clique, we glue a unit $k$--cube, $[ -\frac{1}{2}, \frac{1}{2}]^k$, to $gW_T \subset \Sigma_{\Gamma}$. 

Much like the Cayley graph of $W_{\Gamma}$, we will assume that $1$-cells of $\Sigma_{\Gamma}$ are labeled by letters of $\Gamma$ corresponding to the associated generator. Furthermore, vertices of $\Sigma_{\Gamma}$ are labeled by group elements of $W_{\Gamma}$.

The $1$--skeleton of $\Sigma_{\Gamma}$ inherits the word metric of the Cayley graph of $W_{\Gamma}$. This is known as the combinatorial metric, and it is quasi-isometric to the CAT(0) metric induced by the Euclidean cubes (see \cite{CS} for instance). When we refer to a geodesic in $\Sigma_{\Gamma}$, we always mean this to be a geodesic in the $1$--skeleton with respect to the combinatorial metric.

\subsection{Hyperplanes in $\Sigma_{\Gamma}$}

The following discussion of hyperplanes and cube complexes holds in the much more general setting of CAT(0) cube complexes. We refer the reader to \cite{Wise} for a general reference. For simplicity, we state the relevant definitions and facts in terms of the Davis complex $\Sigma_{\Gamma}$.

A \textit{midcube} $Y \subset C$ is the restriction of a coordinate of a give cube, $C = [ -\frac{1}{2}, \frac{1}{2}]^k$, of $\Sigma_{\Gamma}$ to $0$. A \textit{hyperplane} $\mathcal{H} \subset \Sigma_{\Gamma}$ is a subspace of $\Sigma_{\Gamma}$ with the property that for each cube, $C$, in $\Sigma_{\Gamma}$, $\mathcal{H} \cap C$ is a midcube or $\mathcal{H} \cap C = \emptyset$. $\Sigma_{\Gamma} - \mathcal{H}$ consists of exactly two distinct components. A $1$--cell, $e$, is dual to a hyperplane $\mathcal{H}$ if $e \cap \mathcal{H} \neq \emptyset$. The carrier of a hyperplane, $N(\mathcal{H})$, is the set of all cubes in $\Sigma_{\Gamma}$ which have non-trivial intersection with $\mathcal{H}$. 

Given a hyperplane, $\mathcal{H}$, in $\Sigma_{\Gamma}$, it is readily checked that $1$-cells dual to $\mathcal{H}$ are labeled a common letter $t \in \Gamma$. Accordingly, we say $\mathcal{H}$ is of type $t$. Furthermore, $N(\mathcal{H})$ is isometric to $\Sigma_t \times \Sigma_{Link(t)}$, where $\Sigma_t$ is a $1$-cell labelled by the generator $t$ and $\Sigma_{Link(t)}$ is the Davis complex corresponding to $W_{Link(t)}$. Let $\mathcal{H}$ and $\mathcal{H}'$ be hyperplanes of types respectively $s$ and $s'$. It follows that if $\mathcal{H}$ intersects $\mathcal{H}'$, then $s$ is adjacent to $s'$ in $\Gamma$.

\subsection{Disk diagrams in $\Sigma_{\Gamma}$}

A \textit{disk diagram}, $D$, is a contractible finite $2$--dimensional cube complex with a fixed planar embedding $P: D \to \mathbb{R}^2$. By compactifying $\mathbb{R}^2$, $S^2 = \mathbb{R}^2 \cup \infty$, we can extend $P$ to the map $P: D \to S^2$, giving a cellulation of $S^2$. The \textit{boundary path of} $D$, $\partial D$, is the attaching map of the cell in this cellulation containing $\infty$. Note that this is not necessarily the topological boundary.

$D$ is a \textit{disk diagram in} $\Sigma_{\Gamma}$, if $D$ is a disk diagram and there is a fixed continuous combinatorial map of cube complexes $F: D \to \Sigma_{\Gamma}$. By a lemma of Van Kampen, for every null-homotopic closed combinatorial path $p: S^{1} \to \Sigma_{\Gamma}$, there exists a disk diagram $D$ in $\Sigma_{\Gamma}$ such that $\partial D = p$. 

Suppose $D$ is a disk diagram in $\Sigma_{\Gamma}$ and $t$ is a $1$--cell of $D$. A \textit{dual curve, $H$, dual to $t$} is a concatenation of midcubes in $D$ that contains a midcube that intersects $t$. Every edge in $D$ is dual to exactly one maximal dual curve. The image of $H$ under the map $F: D \to \Sigma_{\Gamma}$ lies in some hyperplane $\mathcal{H} \subset \Sigma_{\Gamma}$.

\subsection{Thick spaces} \label{sec_thick_overview}

This subsection gives an overview of the definitions of a thick and algebraically thick space. The background here will not be necessary until Sections \ref{sec_thick_bounds} and \ref{sec_thick_not_alg}, so the reader may wish to skip this subsection until then.

We work with the ``strong'' thickness definitions from \cite{BD}. As we will never make reference to the weaker notions of thickness, we will drop the word ``strongly'' from our definitions.

$X$ will denote a metric space and $Y \subset X$ a subspace. $Y$ is $C$--\textit{path connected} if for any $y_1, y_2 \in Y$ there exists a path from $y_1$ to $y_2$ in $N_C(Y)$. $Y$ is $(C,L)$--\textit{quasi-convex} if for any $y_1, y_2 \in Y$, there exists an $(L,L)$--quasi-geodesic in $N_C(Y)$ connecting $y_1$ and $y_2$. 

Roughly, $X$ forms a tight network of spaces with respect to the subsets $\{ Y_{\alpha} \}_{\alpha \in A}$ if these subsets coarsely cover $X$. Furthermore, any two subsets can be connected by a sequence of subsets such that consecutive subsets in this sequence coarsely intersect in an infinite diameter set. This is formally defined below.

\begin{definition}[Tight network of subspaces] \cite[Definition 4.1]{BD}	
	
	Given $C>0$ and $L >0$,  $X$ is a $(C, L)$--\textit{tight network with respect to a collection $\{ Y_{\alpha} \}_{\alpha \in \mathcal{A}}$ of subsets} if the following hold: 
	\begin{description}
		\item[a)] Every $Y \in \{ Y_{\alpha} \}_{\alpha \in  \mathcal{A}}$  with the induced metric is $(C, L)$--quasi-convex
		\item[b)] $X = \cup_{\alpha \in  \mathcal{A}}{N_C(Y_{\alpha})}$
		\item[c)] For every $Y, Y' \in \{Y_{\alpha} \}$ and any $x \in X$ such that $N_{3C}(x)$ intersects both $Y$ and $Y'$, there exists a sequence of length $n \le L$
		\[ Y=Y_1, ~Y_2,...,~Y_{n-1},~Y_n = Y'\] with $Y_i \in \{ Y_{\alpha} \}$
		such that for all $1 \le i <n$, $N_C(Y_i) \cap N_C(Y_{i+1})$ is of infinite diameter, $L$--path connected and intersects $N_L(x)$. 
	\end{description}

\end{definition}

A metric space is \textit{wide} if every one of its asymptotic cones has cutpoints, and, additionally, every point in the space is uniformly near to a $(L,L)$--quasi-geodesic. The following definition provides a uniform version of this notion. 

\begin{definition}[Uniformly wide] \cite[Definition 4.11]{BD} A collection of metric spaces, $\{Y_{\alpha} \}_{\alpha  \in A}$, is \textit{$(C,L)$--uniformly wide} if: 
	\begin{enumerate}
		\item There exists $C, L \ge 0$ such that for every $Y \in \{Y_{\alpha} \}_{\alpha  \in A}$ and for every $y \in Y$, $y$ is in the $C$ neighborhood of some bi-infinite $(L, L)$--quasi-geodesic in $Y$. 
		
		\item Given any sequence of metric spaces $(Y_i, d_i)$ in $\{Y_{\alpha} \}$, any ultrafilter $\omega$, any sequence of scaling constants $(s_i)$ and any sequence of basepoints $(b_i)$ with $b_i \in Y_i$, it follows that the ultralimit $\lim_{\omega}{(Y_i, b_i, \frac{1}{s_i}d_i)}$ does not have cut-points.  
	\end{enumerate}
\end{definition}

Metric thickness of a space $X$, defined below, provides an inductive decomposition of $X$ into tight network of spaces. The base case consists of a set of uniformly wide spaces.

\begin{definition}[Metric thickness] \cite[Definition 4.13]{BD}	A family of metric spaces is \textit{$(C,L)$--thick of order zero} if it is $(C,L)$--uniformly wide.
	
	Given $C \ge 0 $ and $k \in \mathbb{N}$ we say that a metric space $X$ is \textit{ $(C, L)$--thick of order at most $k$ with respect to a collection of subsets $\{ Y_{\alpha} \}$ } if 
	
	\begin{enumerate}
		\item $X$ is a $(C, L)$--tight network with respect to $\{ Y_{\alpha} \}$. 
		\item The subsets in $\{ Y_{\alpha} \}$ endowed with the restriction of the metric on $X$ compose a family of spaces that are $(C,L)$--thick of order at most $k-1$. 
	\end{enumerate}
	
	Furthermore, $X$ is said to be \textit{thick of order k} (with respect to $\{ Y_{\alpha} \}$) if it is $(C, L)$--thick of order at most $k$ (with respect to $\{ Y_{\alpha} \}$) and for no choices of $C, L$ and $\{ Y_{\alpha} \}$ is X $(C,L)$--thick of order at most $k-1$. 
\end{definition}

The following definitions give an algebraic version for thickness. The algebraic condition often implies stronger results (see \cite{BD}).

\begin{definition}[Tight algebraic network of subgroups] \cite[Definition 4.1]{BD}	Let $C>0$, $G$ a finitely generated group and $\mathcal{H}$ a set of subgroups of $G$. $G$ is a $C$--\textit{tight algebraic network with respect to $\mathcal{H}$} if the following hold: 
	\begin{description}
		\item[a)] Every $H \in \mathcal{H}$  is $C$--quasi-convex
		\item[b)] The union of all subgroups in $\mathcal{H}$ generates a finite index subgroup of $G$.
		\item[c)] For every $H, H' \in \{\mathcal{H} \}$, there exists a sequence
		\[ H_1 = H, H_2,...,H_{n-1},H_n = H'\] with $H_i \in \{ \mathcal{H} \}$
		such that for all $1 \le i <n$, $H_i \cap H_{i+1}$ is infinite and is $C$--path connected. 
	\end{description}
	
	By \cite[Proposition 4.3]{BD}, if $G$ admits a tight algebraic network of subgroups with respect to $\mathcal{H}$ then $G$ is a tight network of subspaces with respect to the left cosets of groups in $\mathcal{H}$.  
\end{definition}

\begin{definition}[Algebraic thickness] \cite[Definition 4.13]{BD} Let $G$ be a finitely generated group. $G$ is \textit{algebraically thick of order zero} if it is wide. Given $C \ge 0 $, $G$ is \textit{ $C$--algebraically thick of order at most $k$ with respect to a finite collection of subgroups $\mathcal{H}$ } if
	
	\begin{enumerate}
		\item $G$ is a $C$--tight algebraic network with respect to $\mathcal{H}$.
		\item Every $H \in \mathcal{H}$ is algebraically thick of order at most $k-1$.
	\end{enumerate}
	
	G is algebraically thick of order $k$ if it is algebraically thick of order at most $k$ and is not algebraically thick of order $k-1$. 
\end{definition}

\section{Lambda hypergraphs and the hypergraph index} \label{sec_lambda_graphs}

We describe a sequence of hypergraphs associated to a simplicial graph. Using this construction we define the hypergraph index of a right-angled Coxeter group.

\begin{definition}[Wide and strip subgraphs] \label{def_wide_subgraphs}
Let $\Gamma$ be a simplicial graph. Let $\Omega = \Omega(\Gamma)$ denote the set of induced subgraphs of $\Gamma$ such that given $L \in \Omega$, $L = A \star B$ where $A$ and $B$ are induced subgraphs which each contain a pair of non-adjacent vertices. Furthermore, $L$ is maximal in $\Omega$, i.e. if $L \subset L'$ for some $L' \in \Omega(\Gamma)$, then $L = L'$. The subgraphs in $\Omega$ are the \textit{wide subgraphs} of $\Gamma$.  

Let $\Psi = \Psi(\Gamma)$ denote the set of induced subgraphs of $\Gamma$ such that given $L \in \Psi$, $L = A \star K$ where $A$ is a set of two non-adjacent vertices and $K$ is a non-empty clique. Furthermore, we require that if $L \subset L'$ for any  $L' \in \Omega(\Gamma) \cup \Psi(\Gamma)$ then $L = L'$. The subgraphs in $\Psi$ are the \textit{strip subgraphs} of $\Gamma$.
\end{definition}

\begin{remark} By \cite{BFHS}, $\Omega$ characterizes all maximal special subgroups of $\Gamma$ which are wide (see section \ref{sec_thick_overview} for the relevant definition).  The term ``strip subgraphs'' is used since given $L = A \star K \in \Psi$, the Cayley graph of $W_L$ is isometric to $\mathbb{Z} \times Q$, where $Q$ is isometric to a cube of dimension $|W_{K}|$. 
\end{remark}

\begin{remark} \label{rmk_wide_breakdown}
Given any wide subgraph, $L = A \star B \in \Omega(\Gamma)$, $L$ decomposes as $L = A' \star B' \star K$ where $K$ the (possibly empty) set of all vertices in $L$ which are adjacent to every other vertex of $L$. It follows that $K$ is a clique. Note that given any $a_1 \in A'$, there is always some $a_2 \in A'$ such that $a_1$ is not adjacent to $a_2$. These observations will be used throughout the paper.
\end{remark}

\begin{remark} \label{rmk_strip_subgroup_intersection}
	For $L = A \star K, L' = A' \star K' \in \Psi(\Gamma)$ distinct strip subgraphs, it follows that $A \neq A'$. For if $A = A'$, by the maximal property of strip subgraphs, there must be vertices $k \in K$ and $k' \in K'$ such that $k$ and $k'$ are not adjacent in $\Gamma$. Hence, $A \star (K \cup K')$ is contained in some subgraph of $\Omega(\Gamma)$, which is not allowed by the definition of strip subgraphs.
\end{remark}

\begin{definition}[Lambda hypergraphs] \label{def_lambda_hypergraphs}

For each integer $i \ge 0$, we define the hypergraph $\Lambda_i = \Lambda_i(\Gamma)$ inductively. For each $i$, the vertex set of $\Lambda_i$ is $V(\Gamma)$, the same as that of $\Gamma$. 

\begin{enumerate}
\item For every $L \in \Omega(\Gamma) \cup \Psi(\Gamma)$, $V(L)$ is a hyperedge of $\Lambda_0$. 
\item For $H, H' \in \Lambda_i$, set $H \equiv_i H'$ if there are hyperedges 
\[H = H_0, H_1, ..., H_n = H' \in \mathcal{E}(\Lambda_{i})\]
 such that for each $j$, $0 \le j < n$, $H_{j} \cap H_{j+1}$ contains a pair of non-adjacent vertices. A hyperedge of $\Lambda_{i+1}$ is the union of the vertices of a maximal set of pairwise $\equiv_{i}$-equivalent hyperedges of $\Lambda_{i}$. 
\end{enumerate}
\end{definition}

For an example of these hypergraphs, see Figure \ref{fig_hypergraph_example}. The following definition creates a tree poset structure on the set of hyperedges of the lambda hypergraphs.

\begin{definition}[Membership] \label{def_membership}
A hyperedge $H \in \Lambda_i(\Gamma)$ is \textit{a member of} the hyperedge $H' \in \Lambda_{i+1}(\Gamma)$ if $H'$ was constructed from the $\equiv_i$ equivalence class of $H$. Additionally, if there is a sequence of hyperedges
\[ H_1 \in \Lambda_i(\Gamma),~ H_2 \in \Lambda_{i+1}(\Gamma),~... ,~H_n \in \Lambda_{i+n-1}(\Gamma) \]
such that for $1 \le j < n$, $H_j$ is a member of $H_{j+1}$, then we also say $H_1$ is a member of $H_n$. It follows that given integers $j > i \ge 0$, and a hyperedge, $H \in \mathcal{E}(\Lambda_i)$, then $H$ is a member of an unique hyperedge of $\Lambda_{i+j}$.
\end{definition}

\begin{remark} \label{rmk_membership_subtlety}
	A subtle point of the membership definition is that given a hyperedge $H \in \Lambda_i$, all vertices of $H$ may be contained in the hyperedge $H' \in \Lambda_{i+1}$, but $H$ does not necessarily have to be a member of $H'$.
	
	Consider, for instance, the graph in Figure \ref{fig_membership_example}. $H = \{v_1, v_2, v_3, v_4\}$ is a hyperedge of $\Lambda_{0}$. Note also that the intersection of $H$ with any other hyperedge of $\Lambda_0$ does not contain a pair of non-adjacent vertices. It follows there is a hyperedge $Z = \{v_1, v_2, v_3, v_4\}$ of $\Lambda_1$. 
	
	It can also be readily checked, that some other hyperedge $Z'$ of $\Lambda_{1}$ contains every vertex of $\Gamma$. In particular, the vertices $v_1, v_2, v_3$ and $v_4$ are also in $Z'$. However, $H$ is a member of $Z$, but $H$ is \textit{not} a member of $Z'$. 
\end{remark}

\begin{figure}[htp]
	\centering
	\begin{overpic}[scale=.55]{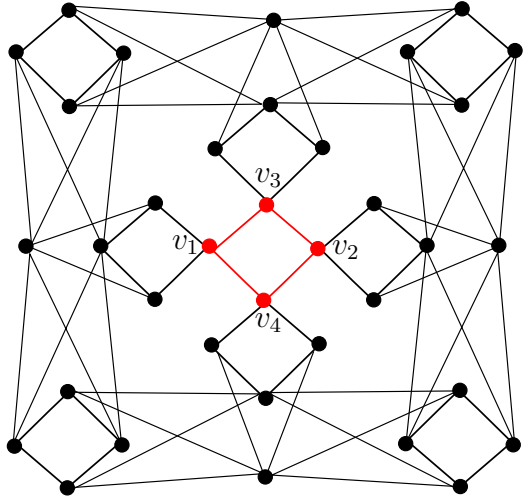}

		\put(32,48){$v_1$}
		\put(63,47){$v_2$}
		\put(48,61){$v_3$}
		\put(48,33){$v_4$}
		
	\end{overpic}
	
	\caption{A graph illustrating Remark \ref{rmk_membership_subtlety} } \label{fig_membership_example}
\end{figure}

A pair of non-adjacent vertices uniquely determines a hyperedge in $\Lambda_1(\Gamma)$:

\begin{lemma} \label{lemma_unique_membership}
	Let $s, t \in \Gamma$ be two non-adjacent vertices. There is an unique hyperedge, $Z$, of $\Lambda_1(\Gamma)$ such that any hyperedge of $\Lambda_0(\Gamma)$ that contains $s$ and $t$ is a member of $Z$. 
\end{lemma}
\begin{proof}
	Let $H$ be a hyperedge of $\Lambda_0$ containing $s$ and $t$. $H$ is a member of some unique $Z \in \Lambda_1(\Gamma)$ formed by the $\equiv_0$ equivalence class of $H$. Suppose some other hyperedge of $\Lambda_0$, $H'$, contains $s$ and $t$. By definition, $H$ and $H'$ are in the same $\equiv_0$ equivalence class. Thus, $Z$ is well-defined and unique.
\end{proof}

Given a hyperedge $H$ of $\Lambda_i$, we define $W_{H}$ as the special subgroup of $W_{\Gamma}$ induced by the vertices of $H$.

\begin{definition}[Hypergraph Index] \label{def_hypergraph_index}
$\Gamma$ has \textit{hypergraph index} $h \in \mathbb{N}$, if some hyperedge in $\Lambda_h(\Gamma)$ contains every vertex of $\Gamma$ and no hyperedge of $\Lambda_{h-1}(\Gamma)$ contains every vertex of $\Gamma$. Additionally, it is required that the set of wide subgraphs, $\Omega(\Gamma)$, is not empty. If there is no such $h$ or $\Omega(\Gamma)$ is empty, then we say $\Gamma$ has infinite hypergraph index. The \textit{hypergraph index of a right-angled Coxeter group, $W_{\Gamma}$}, is the hypergraph index of $\Gamma$.
\end{definition}

\begin{remark} \label{rmk_infinite_hyp_index}
It is not difficult to show, given the results of \cite{BHSC}, that $\Gamma$ has hypergraph index $h = \infty$ if and only if $W_{\Gamma}$ is relatively hyperbolic.
\end{remark}

We define the \textit{realization} of $\Lambda_i(\Gamma)$. These are cosets of special subgroups of $W_{\Gamma}$ corresponding to hyperedges of $\Lambda_i(\Gamma)$, but excluding hyperedges corresponding to strip subgroups. 

\begin{definition}
	The realization $\mathcal{R}_i = \mathcal{R}_i(\Gamma)$ of a graph $\Gamma$ is the set of cosets 	
	\[ \mathcal{R}_i =  \{ gW_{H} \subset W_{\Gamma} ~ | ~ H \text{ is a hyperedge of } \Lambda_i(\Gamma),~ H \notin \Psi(\Gamma), ~ g \in W_{\Gamma} \} \]
	Recall $\mathcal{H}(\Lambda_i(\Gamma))$ is the set of hyperedges of $\Lambda_i(\Gamma)$. By $H \notin \Psi(\Gamma)$, we mean that the subgraph of $\Gamma$ induced by vertices of $H$ is not in $\Psi(\Gamma)$. We often think of the cosets in $\mathcal{R}_i$ as geometric subsets of the Davis complex $\Sigma_{\Gamma}$.
\end{definition}

We extend the membership definition to the realization cosets as follows. Given $hW_{H} \in \mathcal{R}_i$ and $h'W_{H'} \in \mathcal{R}_{i+j}$, $hW_{H}$ is a \textit{member of} $h'W_{H'}$ if $H$ is a member of $H'$ and $hW_{H} \subset h'W_{H'}$. It readily follows that if $hW_{H}$ is a member of $h'W_{H'}$ then the coset representatives can be chosen such that $h = h'$.

Recall that two subsets $A$, $B$ of some metric space are $C$--Hausdorff close if $B \subset N_C(A)$ and $A \subset N_C(B)$. $A$ and $B$ are Hausdorff close if they are $C$--Hausdorff close for some $C \ge 0$. The following lemma shows that distinct cosets in $\mathcal{R}_i$ are not Hausdorff close.

\begin{lemma} \label{lemma_equal_cosets}
	Let $L_1$ and $L_2$ be cosets in the realization $\mathcal{R}_i(\Lambda)$. If $L_1$ and $L_2$ are Hausdorff close as subsets of $\Sigma_{\Gamma}$, then $L_1 = L_2$.
\end{lemma}
\begin{proof}
	Suppose $L_1$ and $L_2$ are $C$--Hausdorff close for some $C>0$. Without loss of generality, let $L_1 = W_{H_1}$ and $L_2 = h_2W_{H_2}$ with $H_1, H_2$ hyperedges of $\Lambda_i(\Gamma)$ and $h_2 \in W_{\Gamma}$. Set $H_1 = G_1 \star K_1$ and $H_2 = G_2 \star K_2$, where $K_1$ contains every $s \in H_1$ which is adjacent to every other vertex in $H_1$. Similarly, $K_2$ contains every $s \in H_2$ which is adjacent to every other vertex in $H_2$. $K_1$ and $K_2$ are (possibly empty) cliques.
	
	Fix $s \in G_1$. It follows there is a $s' \in G_1$ such that $s$ and $s'$ are not adjacent in $\Gamma$ (since $s \notin K_1$). The sequence of vertices, $\{s, ss',ss's, ss'ss',... \}$, and the edges connecting consecutive vertices in this sequence, forms a ray, $\beta$, in the Davis complex $\Sigma_{\Gamma}$. $\beta$ is geodesic by Tit's solution to the word problem (see \cite{Dav}) and is contained in $W_{H_1}$.
	
	Choose vertices $x, y \in \beta$ such that $d(x,y) \ge 4C + 2$. Let $\beta'$ be the segment of $\beta$ from $x$ to $y$. Let $\alpha_1$ be a geodesic from $x$ to some vertex $x' \in L_2$ and $\alpha_2$ be a geodesic from $y$ to some vertex $y' \in L_2$. We can choose $\alpha_1$ and $\alpha_2$ so that $|\alpha_1|, |\alpha_2| \le C$. Let $\gamma$ be a geodesic from $x'$ to $y'$ contained in $L_2$ (this is possible since $W_{H_2}$ is convex). Let $D$ be a disk diagram with boundary $\beta' \alpha_2 \gamma^{-1} \alpha_1^{-1}$.
	
	There are exactly $2C + 1$ occurrences of the letter $s$ in $\beta'$. Furthermore, at most $2C$ curves dual to $\beta'$ in $D$ can intersect $\alpha_1 \cup \alpha_2$. It follows that some curve dual to an edge of $\beta'$ labeled by $s$ must intersect $\gamma$. Since $\gamma$ is contained in  $L_2$, it follows that $s \in H_2$. As $s$ was an arbitrarily chosen letter of $G_1$, it follows that $G_1 \subset G_2$. By repeating the argument and switching the roles of $G_1$ and $G_2$, we conclude that $G_1 = G_2$.
	
	We next show $H_1$ and $H_2$ are the same. Let $A \in \Omega(\Lambda)$ be a member of $H_1$. By Remark \ref{rmk_wide_breakdown}, we can write $A = (A_1 \star A_2) \star K_3$ where $K_3$ contains all vertices of $A$ which are adjacent to every other vertex of $A$. It follows that $A_1 \star A_2 \subset G_1 = G_2$. We will now show that $K_1 \cup K_2 \subset A$. Let $k \in K_1 \cup K_2$ and suppose $k \notin A$. Since $H_1 = G_1 \star K_1$ and $H_2 = G_2 \star K_2$, $k$ is adjacent to every vertex of $A_1 \star A_2$. However, this implies that $(A_1 \cup k \cup K_3) \star A_2$ is a wide subgraph, contradicting the maximality of $A$ as a wide subgraph. Hence, $K_1 \cup K_2 \subset A$ and, consequently, $K_1 \cup K_2 \subset H_1$ as $A$ is a member of $H_1$. We can similarly conclude that $K_1 \cup K_2 \subset H_2$. It follows that $H_1 = H_2$. For the remainder of the proof, set $H = H_1 = H_2$.
	
	To show $L_1 = L_2$, we need to show the coset representative of $L_2$ can be chosen to be the identity. For a contradiction, assume there is some generator, $s$, in a minimal expression for $h_2$ such that $s \notin H$. Let $\mathcal{H}$ be a hyperplane in $\Sigma_{\Gamma}$ through the letter $s$ in $h_2$. $\mathcal{H}$ separates $W_{H}$ from $h_2 W_{H}$. 
	
	Suppose $A = (A_1 \star A_2) \star K \in \Omega(\Lambda)$ is a member of $H$, where $K$ contains every vertex of $A$ that is adjacent to every other vertex of $A$. There must be some $a \in A_1 \star A_2$ such that $a$ and $s$ are not adjacent in $\Gamma$. For if not, $(A_1 \cup s \cup K) \star A_2$ is a wide subgraph, contradicting the maximality of $A$. Let $b \in A_1 \star A_2$ be such that $b$ is not adjacent to $a$. Consider the infinite ray formed by concatenating the vertices $(ab)^n$, for $n \in \mathbb{N}$, which is contained in $W_H$. Every hyperplane dual to an edge labelled by $a$ cannot cross $\mathcal{H}$. Since any path from $(ab)^n$ to $\mathcal{H}$ must cross each of these pairwise non-intersecting hyperplanes, it follows that $d(\mathcal{H}, (ab)^n) \ge n$. However, this implies there are points in $L_1$ which are arbitrarily far from $L_2$, a contradiction.	
\end{proof}

An important consequence of the next lemma is that if a neighborhood of a coset in $\mathcal{R}_i$ intersects another coset in $\mathcal{R}_i$ in an infinite diameter set, then both these cosets are members of a common coset of $\mathcal{R}_{i+1}$. 

\begin{lemma} \label{lemma_infinite_intersection}
	Let $L_1$ and $L_2$ be cosets in $\mathcal{R}_i(\Gamma)$, thought of as subsets of $\Sigma_{\Gamma}$. Suppose for some $C>0$, $N_C(L_1) \cap L_2$ is an infinite diameter subset of $\Sigma_{\Gamma}$, then the following are true:
	
	\begin{enumerate}
		\item 
		Either $L_1 \cap L_2$ has infinite diameter or there is some $H \in \Omega(\Gamma) \cup \Psi(\Gamma)$ and $h \in W_{\Gamma}$ such that $L_1 \cap hW_H$ and $L_2 \cap hW_H$ each have infinite diameter.
		
		\item 
		There is some coset $L_3 \in \mathcal{R}_i(\Gamma)$ such that $N_M(L_1) \cap L_3$ and $N_M(L_2) \cap L_3$ each have infinite diameter, where $M$ is a constant one larger than the maximal clique size of $\Gamma$.
		\item 
		$L_1, L_2$ and $L_3$, as above, are all members of a common coset in $\mathcal{R}_{i+1}$. 
	\end{enumerate} 
\end{lemma}

\begin{proof}
	We start by proving the first statement. Let $M$ be one larger than the maximal clique size in $\Gamma$. Let $\alpha_1$ and $\alpha_2$ be geodesics in the $1$--skeleton of $\Sigma_{\Gamma}$ from $L_1$ to $L_2$ such that $|\alpha_1|, |\alpha_2| \le C$ and such that $\alpha_1$ and $\alpha_2$ are at a distance at least $2C + (2C+1)M$ apart. Let $f_1$ be a geodesic in $L_1$ from the start of $\alpha_1$ to the start of $\alpha_2$, and let $f_2$ be a geodesic in $L_2$ from the endpoint of $\alpha_1$ to the endpoint of $\alpha_2$. These geodesics exist since special subgroups are convex. Let $D$ be a disk diagram with boundary $f_1\alpha_2f_2^{-1}\alpha_1^{-1}$.
	
	By \cite[Lemma 2.6]{Wise}, we can choose $D$ and $f_1$ such that two distinct curves dual to $f_1$ in $D$ do not intersect each other. At most $2C$ curves dual to $f_1$ can intersect $\alpha_1 \cup \alpha_2$. It follows there are $(2C+1)M$ consecutive curves dual to $f_1$ which intersect $f_2$. At most $2C$ dual curves to $f_2$ can intersect $\alpha_1 \cup \alpha_2$. By the pigeonhole principle, there must be a set of $M$ consecutive dual curves to $f_1$ which is also a set of $M$ consecutive dual curves to $f_2$. It follows there is a subdiagram, $D'$, of $D$ which is isometric to an euclidean rectangle connecting with opposite sides on $f_1$ and $f_2$. $D' = f_1' \times b$ where $f_1'$ is a subpath of $f_1$ and $b$ is a geodesic from $f_1'$ to $f_2$.  Additionally, it follows that $|f_1'| = M$. 
	
	Let $A$ be the set of generators in $\Gamma$ which appear as a letter of $f_1'$ and $B$ the set of generators of $\Gamma$ which appear as a letter of $b$. Since $f_1'$ is a geodesic and its length is larger than the maximal clique size in $\Gamma$, it follows by Tits solution to the word problem (see \cite{Dav}) that $A$ must contain two non-adjacent vertices, say $a_1$ and $a_2$. Set $A' = \{a_1, a_2 \}$.
	
	Let $h$ be a geodesic in $\Sigma_{\Gamma}$ from the identity element to the start point of $f_1'$. If $B$ is empty ($b$ is a vertex), it follows that the vertex $h(a_1a_2)^n$ is contained in both $L_1$ and $L_2$ for all positive integers $n$. Thus, claim 1 is true for this case. On the other hand, consider the case where $B$ is nonempty.  In this case, $h(a_1a_2)^n$ is contained in $L_1$ and $hb(a_1a_2)^n$ is contained in $L_2$ for all positive integers $n$. It follows that $A' \star B$ is a subgraph of some maximal graph $H \in \Omega(\Gamma) \cup \Psi(\Gamma)$. Furthermore, $hW_H$ intersects both $L_1$ and $L_2$ in an infinite diameter set. Hence, claim 1 also follows for this other case as well. 
	
	We now show that claim 1 implies claim 2. If $H \in \Psi(\Gamma)$ then claim 2 follows by setting $L_3 = L_1$. Otherwise, if $H \in \Omega(\Gamma)$, then $hW_{H}$ is a member of some $L_3 \in \mathcal{R}_{i}$ and claim 2 follows.	
	
	We now show that claim 3 follows from claim 2. Consider first the following fact. Let $G_1$ and $G_2$ be induced subgraphs of $\Gamma$ and $g_1, g_2 \in W_{\Gamma}$. Furthermore, suppose that $g_1W_{G_1} \cap g_2W_{G_2}$ has infinite diameter. As $W_{G_1}$ and $W_{G_2}$ are convex, $g_1W_{G_1} \cap g_2W_{G_2}$ contains a geodesic of length greater than $M$. By Tit's solution to the word problem, it follows that there must be two non-adjacent vertices $s, t \in \Gamma$ that appear as edges of this geodesic. Hence, $G_1 \cap G_2$ contains two non-adjacent vertices, namely $s$ and $t$. 
	
	Set $L_1 = h_1W_{H_1}$ and $L_2 = h_1W_{H_2}$. By the proof of claim 2, there is an induced subgraph $H \in \Psi(\Gamma) \cup \Omega(\Gamma)$ and $h \in \Gamma$ such that $L_1 \cap hW_H$ and $hW_H \cap L_2$ are infinite. By the above fact, $H_1 \cap H$ and $H_2 \cap H$ each contain two non-adjacent vertices. Thus, $H, H_1$ and $H_2$ are in a common hyperedge of $\Lambda_{i+1}$. 
	
	If $H \in \Psi(\Gamma)$ we had set $L_3 = L_1$. Otherwise, $H \in \Omega(\Gamma)$ and we found $L_3 \in \mathcal{R}_{i}$ containing the member $hW_H$. In either case, $L_1, L_2$ and $L_3$ are all members of a common coset in $\mathcal{R}_{i+1}$ as  $H, H_1$ and $H_2$ are in a common hyperedge of $\Lambda_{i+1}$. Claim 3 then follows.
\end{proof}

\section{Coarse intersection degree} \label{sec_coarse_intersection}

In this section we define the coarse intersection degree of a collection of subspaces, $\mathcal{M}$, of a metric space. The coarse intersection degree is closely related to the notion of a tight network, defined in section \ref{sec_thick_overview}. We then explore the relationship between the coarse intersection degree of the realization $\mathcal{R}_i(\Gamma)$ and the hypergraph index of $\Gamma$.

\begin{definition} \label{def_coarse_intersection}
	Let $X$ be a metric space and $\mathcal{M}$ a collection of subspaces of $X$. The \textit{coarse intersection degree} of $\mathcal{M}$ is the smallest integer $d$ such that there are collections of subspaces $\{\mathcal{M}=\mathcal{M}_0,~ \mathcal{M}_1,~ ...,~ \mathcal{M}_d \}$ and a constant $C>0$ satisfying: 
	\begin{enumerate}
		\item Given $A \in \mathcal{M}_i$, there is a collection of elements of $\mathcal{M}_{i-1}$, $\{ B_j \}_{j \in J}$, such that
		${A \subset \bigcup_{j \in J}{ N_C( B_{j} )}}$ and ${\bigcup_{j \in J}{ B_{j}} \subset A}$. Given $j \in J$, we refer to $B_j$ as a \textit{piece of} $A$.		
		\item If $A, A' \in \mathcal{M}_{i-1}$ are pieces of some $B \in \mathcal{M}_{i}$, then there is a sequence of pieces of $B$, $A=A_1,..., A_m = A'$, such that $N_C(A_j) \cap A_{j+1}$ has infinite diameter and is path connected for $1 \le j < m$.
		\item There is some $A \in \mathcal{M}_d$ such that given any $B \in \mathcal{M} = \mathcal{M}_0$, it follows that $B \subset N_C(A)$. 
	\end{enumerate}
	We call $C$ a \textit{coarse intersection constant for $\mathcal{M}$} and $\{\mathcal{M}=\mathcal{M}_0,~ \mathcal{M}_1,~ ...,~ \mathcal{M}_d \}$ a \textit{coarse intersection sequence}.
\end{definition}

\begin{remark} \label{rmk_coarse_intersection_qi}
Suppose $\phi: X \to Y$ is a quasi-isometry, and $\mathcal{M}$ is a collection of subspaces of $X$. Let $D>0$ be any constant. The coarse intersection degree of $\mathcal{M}$ is the same as that of the collection of subspaces, $\{ N_D(\phi(A)) | A \in \mathcal{M} \}$, of $Y$.
\end{remark}

\begin{remark} \label{rmk_coarse_intersection_of_sequence}
Suppose $\{\mathcal{M}=\mathcal{M}_0,~ \mathcal{M}_1,~ ...,~ \mathcal{M}_d \}$ is a coarse intersection sequence for $\mathcal{M}$. It follows that the coarse intersection degree of $\mathcal{M}_i$ is $d - i$.
\end{remark}

\begin{lemma} \label{lemma_coarse_intersection_degree}
	Let $\Gamma$ be a simplicial graph with hypergraph index $h < \infty$, then the coarse intersection degree of the realization $\mathcal{R}_0(\Gamma)$, regarded as a collection of subspaces of the Davis complex $\Sigma_{\Gamma}$, is $h$.
\end{lemma}
\begin{proof}
We show that $\{\mathcal{R}_0, \mathcal{R}_1,...,\mathcal{R}_h \}$ satisfies definition \ref{def_coarse_intersection} where the pieces of an element of $\mathcal{R}_i$ are its members in $\mathcal{R}_{i-1}$. Let $M$ be a constant one larger than the maximal clique size of $\Gamma$. We use $M$ as the coarse intersection constant. 

We first show that with these choices the first criteria of definition \ref{def_coarse_intersection} is satisfied.  To see this, let $A \in \mathcal{R}_i$ and $g$ a group element in $A$. Either $g \in B$ for some member of $A$, $B \in \mathcal{R}_{i-1}$, or $g$ lies on a coset $fW_F$ where $F$ is a strip subgraph and $f \in W_{\Gamma}$. In the latter case, by the definition of the Lambda hypergraphs and Remark \ref{rmk_strip_subgroup_intersection}, there is some $F' \in \Lambda_{i-1}$ such that $F \cap F'$ contains the two non-adjacent vertices of $F$. It follows that  $g \in fW_F \subset N_M(fW_{F'})$. Hence $A$ is contained in the $M$ neighborhood of its members. Furthermore, every member of $A$ is contained in $A$. The first criteria thus holds. 

We now show the third criteria of definition \ref{def_coarse_intersection} is satisfied. Since $h$ is the hypergraph index of $\Gamma$, by definition some hyperedge $H$ of  $\Lambda_h$ contains every vertex of $\Gamma$. It then follows that $W_H = W_{\Gamma} \in \mathcal{R}_h$. Hence the third criteria is satisfied.

We now prove the second criteria of definition \ref{def_coarse_intersection}. Given $n$, such that $1 \le n \le h$, let $A = fW_F \in \mathcal{R}_{n}$ and $B_1, B_2 \in \mathcal{R}_{n-1}$ any two members of $A$. Let $p \in B_1$ and $q \in B_2$ be points in the $1$--skeleton of $\Sigma_{\Gamma}$. Let $\alpha = s_1s_2...s_m$, for $s_j \in V(\Gamma)$, be a geodesic connecting $p$ and $q$ which lies in $A$ (this is possible by the convexity of special subgroups, Lemma \ref{lemma_special_subgroups}).

For $t < m$ there is a sequence of hyperedges $G_0^t, G_1^t, ...,G_l^t \in \Lambda_{n-1}$ such that $G_j^t$ is a member of $F$ (where $A = fW_F$), $G_j^t \cap G_{j+1}^t$ contains a pair of non-adjacent vertices, $s_t \in G_0^t$ and $s_{t+1} \in G_l^t$. Furthermore, we can choose these sequences such that $G_l^t = G_0^{t+1}$. (We note for later use that $l$ is bounded by a constant only depending on $\Gamma$.) 

Let $G_{i_0}^t, G_{i_1}^t, ..., G_{i_k}^t$ be the subsequence of $G_0^t, G_1^t, ...,G_l^t$ obtained by deleting elements which are strip subgraphs. By Remark \ref{rmk_strip_subgroup_intersection}, no two consecutive elements are deleted. Furthermore, for $j < k$, $N_M(W_{G_{i_j}}) \cap W_{G_{i_{j+1}}}$ has infinite diameter and is path connected.

Let $g$ be a geodesic from the identity element to $p$. For each $t < m$, we can define the sequence of members of $A$:
\[S_t = \{gs_1...s_{t-1}G_{i_0}^t, gs_1...s_tG_{i_1}^t, gs_1...s_tG_{i_2}^t, ..., gs_1...s_tG_{i_k}^t \}\]

The concatenation of the sequences: $\{S_1, S_2, ..., S_{l-1}\}$ satisfies the second criteria of definition \ref{def_coarse_intersection}.

We now claim that $h$ is the smallest integer satisfying definition \ref{def_coarse_intersection}. Suppose we have a sequence $\{ \mathcal{R}_0 = \mathcal{M}_0, ..., \mathcal{M}_d\}$ satisfying definition \ref{def_coarse_intersection} with coarse intersection constant $C>0$ and $d$ minimal. Without loss of generality we may assume $C > M$. We induct on $i$, for $0 \le i \le h$, to show there is a constant, $D > 0$, such that given any $A \in \mathcal{M}_i$, then $A \subset N_D(R)$ for some $R \in \mathcal{R}_i$. This fact will then imply the claim, since $\mathcal{R}_0 \subset N_{D'}(R)$ (condition 3) for some $R \in \mathcal{R}_d$ and $D'>0$, only when $d \ge h$. To see this, suppose $R = W_H \in \mathcal{R}_d$ with $d<h$ and that $\mathcal{R}_0 \subset N_{D'}(W_H)$. Since $d < h$, there is some vertex $v \in \Gamma$ not contained in $H$. Furthermore, $v$ is not adjacent to some letter $u \in \Gamma$ (or else $v$ is in every wide subgraph and thus in $H$). Let $g = uv$. Given an integer $n > D'$, $W_H$ and $g^nW_H$ are distance at least $n$ apart, since every hyperplane intersecting $g^n$ at an edge labelled by $v$ cannot intersect $W_H$ or $g^nW_H$. However, this implies some cosets in $\mathcal{R}_0$ are not contained in $N_{D'}(W_H)$, a contradiction.

The base case of the induction, $i=0$, trivially holds as $\mathcal{M}_0 = \mathcal{R}_0$. Now, assume $i>0$ and the claim is true for $i-1$. Let $A \in \mathcal{M}_i$ and $\mathcal{B} \subset \mathcal{M}_{i-1}$ be the collection of pieces of $A$. Fix $B \in \mathcal{B}$. By the induction hypothesis, there is some $R \in \mathcal{R}_{i-1}$ and some constant $D' > 0$, such that $B \subset N_{D'}(R)$. Let $\hat{R} \in \mathcal{R}_i$ be the unique coset that $R$ is a member of.

Given any other piece $B' \in \mathcal{B}$ of $A$, it follows there is a sequence of pieces of $A$:
\[B = B_1, B_2, ..., B_m = B' \] 
such that $N_C(B_j) \cap B_{j+1}$ has infinite diameter. By the induction hypothesis we then get a sequence of cosets in $\mathcal{R}_{i-1}$:
\[R = R_1, R_2,...., R_m \]
such that $B_j \subset N_{D'}(R_j)$. Furthermore, $N_{(C+2D')}(R_j) \cap R_{j+1}$ has infinite diameter. By Lemma \ref{lemma_infinite_intersection}, this implies $R_m$ is also a member of $\hat{R}$. Hence, $\mathcal{B} \subset N_{D'}(\hat{R})$. Since $A \subset N_C(\mathcal{B})$, it follows that $A \subset N_{(D' + C)}(\hat{R})$. We then just set $D = D' + C$, and the induction step holds.
\end{proof}

The following corollary immediately follows from Lemma \ref{lemma_coarse_intersection_degree} and Remark \ref{rmk_coarse_intersection_of_sequence}.

\begin{corollary}
	Let $h$ be the hypergraph index of a simplicial graph $\Gamma$. The coarse intersection degree of $\mathcal{R}_i$ is $h - i$.
\end{corollary}

\section{A QI-invariant for 2-dimensional right-angled Coxeter groups} \label{sec_qi}
A graph is triangle-free if it does not contain any $3$--cycles. The goal of this section is to prove the following theorem:

\begin{theorem} \label{thm_qi_invariant}
	Let $\Gamma$ and $\Gamma'$ be triangle-free graphs. If the right-angled Coxeter groups $W_{\Gamma}$ and $W_{\Gamma'}$ are quasi-isometric, then $W_{\Gamma}$ and $W_{\Gamma'}$ have the same hypergraph index.
\end{theorem} 

Although the hypergraph index of a relatively hyperbolic right-angled Coxeter group is always infinite, we can still define a more refined quasi-isometry invariant for these groups. Given a right-angled Coxeter group, $W_{\Gamma}$, let $\{W_{T_1}, W_{T_2}, ..., W_{T_n}\}$ be the collection of maximal special subgroups of $W_{\Gamma}$ with finite hypergraph index (i.e., if $W_T$ is a special subgroup with finite hypergraph index, then $W_T \subset W_{T_i}$ for some unique $i$). Define the \textit{hypergraph index spectrum of $W_{\Gamma}$} as the set of unique hypergraph indexes, $\mathcal{H} = \{ h_1,h_2,...,h_m \}$, of the groups $\{W_{T_1}, W_{T_2}, ..., W_{T_n}\}$. Note that when $W_{\Gamma}$ is not relatively hyperbolic, the hypergraph index spectrum and the hypergraph index coincide.

\begin{corollary} \label{cor_hypergraph_index_spectrum}
Let $\Gamma$ and $\Gamma'$ be triangle-free graphs. If the right-angled Coxeter groups $W_{\Gamma}$ and $W_{\Gamma'}$ are quasi-isometric, then $W_{\Gamma}$ and $W_{\Gamma'}$ have the same hypergraph index spectrum.
\end{corollary}
\begin{proof}
By \cite[Theorem I]{BHSC} every relatively hyperbolic right-angled Coxeter group is relatively hyperbolic with respect to its maximal special subgroups of finite hypergraph index. By \cite[Theorem 4.8]{BDM} and Theorem \ref{thm_qi_invariant}, such a quasi-isometry induces a bijective function, taking identity values, between the hypergraph index spectrums of $W_{\Gamma}$ and $W_{\Gamma'}$.
\end{proof}

For the remainder of this section and this section only, we assume $\Gamma$ and $\Gamma'$ are triangle-free simplicial graphs associated to the right-angled Coxeter groups $W_{\Gamma}$ and $W_{\Gamma'}$ with corresponding Davis complexes $\Sigma_{\Gamma}$ and $\Sigma_{\Gamma'}$. We assume there exists a quasi-isometry:
\[\phi: W_{\Gamma} \to W_{\Gamma'} \] 

We also fix the following hypergraphs and their realizations as defined in section \ref{sec_lambda_graphs}:
\[\Lambda_i = \Lambda_i(\Gamma),~ \Lambda_i' = \Lambda_i(\Gamma'),~ \mathcal{R}_i = \mathcal{R}_i(\Gamma),~ \mathcal{R}_i' = \mathcal{R}_i(\Gamma')\]
 
A \textit{flat}, $F$, is the image of an isometric embedding of $\mathbb{R}^2$. When we consider a flat in a Davis complex, it is implied this is a flat with respect to the CAT(0) metric. The next two lemmas describe the behavior of flats in $\Sigma_{\Gamma}$. The first of which shows a flat must be contained in a coset of $\mathcal{R}_0$, while the second lemma shows cosets of $\mathcal{R}_0$ contain many flats. 

\begin{lemma} \label{lemma_flat_structures}
	Let $F$ be a flat in $\Sigma_{\Gamma}$. $F$ is contained in a coset of $W_{A \star B}$, where $A$ and $B$ are each a set of pairwise non-adjacent vertices of $\Gamma$ of size at least $2$. Hence, $F \subset P$ for some $P \in \mathcal{R}_0$.
\end{lemma}
\begin{proof}
	First observe that $F$ must contain some point in the interior of a $2$--cell $C$ of $\Sigma_{\Gamma}$, and, consequently, $F$ must contain all of $C$. Since $F$ is isometric to the Euclidean plane, it follows that for each vertex, $v \in C$, $F$ contains exactly four $2$--cells adjacent at $v$ which together form a square composed of a grid of four smaller squares. 
	
	Repeating this process, we can deduce that $F$ is exactly a subcomplex of $\Sigma_{\Gamma}$ which is the product of two combinatorial bi-infinite geodesics $g_1$ and $g_2$. Let $A$ be the set of generators which appear in $g_1$ and $B$ the set of generators which appear in $g_2$. Since $g_1$ and $g_2$ are infinite, it follows $|A|, |B| \ge 2$. Furthermore, since a hyperplane dual to an edge of $g_1$ must cross every hyperplane through $g_2$, and vice versa, it follows every generator in $A$ commutes with every generator in $B$. Hence, $F \subset gW_{A \star B}$ for some $g \in W_{\Gamma}$. No pair of vertices in $A$ are adjacent since $\Gamma$ is triangle-free. Similarly, $B$ contains no pair of adjacent vertices.	
\end{proof}

\begin{remark} \label{rmk_basic_flats}
	Let $F = gW_G$ where $G$ is an embedded $4$--cycle. It follows $G = A \star B \subset \Gamma$ where $A$ and $B$ each consist of two non-adjacent vertices, say $A = \{a_1, a_2\}, B = \{b_1, b_2\}$. The Cayley graph of $W_G$ is the product $g_1 \times g_2$, where $g_1$ is the bi-infinite geodesic $...a_1a_2a_1a_2...$ and $g_2$ is the bi-infinite geodesic $...b_1b_2b_1b_2...$. As special subgroups are convex in right-angled Coxeter groups, $F$ is a flat in $W_{\Gamma}$.
\end{remark}

\begin{lemma} \label{lemma_finding_flats}
Given $P \in \mathcal{R}_0$, considered as a subset of $\Sigma_{\Gamma}$, and a vertex $p \in P$, there is an embedded $4$-cycle, $G \subset \Gamma$, and $g \in W_{\Gamma}$ such that $F = gW_G$ is a flat, $F \subset P$ and $p \in F$. 
\end{lemma}
\begin{proof}
Let $P = W_H$ (we may pick the identity element as the coset representative without loss of generality) where $H = A \star B$. As $P \in \mathcal{R}_0$, $A$ must contain some pair of non-adjacent vertices, say $a_1$ and $a_2$. Similarly, $B$ must contain some pair of non-adjacent vertices, say $b_1$ and $b_2$. Let $G$ be the subgraph induced by $a_1, a_2, b_1$ and $b_2$. By Remark \ref{rmk_basic_flats}, $W_G \subset W_{\Gamma}$ is an embedded isometric copy of $\mathbb{R}^2$. Let $g$ be a geodesic word from the identity element to $p$ which is contained in $W_H$ (this is possible since $W_H$ is convex). It follows that $F = gW_{G}$ is a flat satisfying the desired properties.  
\end{proof}

The following theorem from \cite{Huang} is stated in terms of our given setting. It is an important ingredient in the proof of Theorem \ref{thm_qi_invariant}.

\begin{theorem}[{\cite[Corollary 5.18]{Huang}}] \label{thm_huang} There is a constant $C>0$, depending only on the quasi-isometry constants of $\phi$, such that given any flat $F \subset \Sigma_{\Gamma}$, there is a flat $F' \subset \Sigma_{\Gamma'}$ which is $C$--Hausdorff close to $\phi(F)$.
\end{theorem}

\begin{definition} \label{def_phi_0}
Define the map 
\[ \phi_0: \mathcal{R}_0 \to \mathcal{R}_1' \]
as follows. Given, $P \in \mathcal{R}_0$, let $F$ be a flat contained in $P$ ($F$ exists by Lemma \ref{lemma_finding_flats}). Let $F'$ be some flat which is Hausdorff close to $\phi(F)$ (such a flat always exists by Theorem \ref{thm_huang}). By Lemma \ref{lemma_flat_structures} there is some $P' \in \mathcal{R}_0'$ which contains $F'$. Furthermore, there is a unique $Z' \in \mathcal{R}_1'$ of which $P'$ is a member of (uniqueness follows by Lemma \ref{lemma_unique_membership}).
\end{definition}

For the remainder of the section, fix the map $\phi_0$ as above. The following lemma shows, amongst other facts, that $\phi_0$ is well-defined.

\begin{lemma} \label{lemma_lambda_1_images}
	$\phi_0$ is well-defined. Furthermore, the following is true. Let $D = \max(C, M)$ where $C$ is the constant in Theorem \ref{thm_huang} and $M$ is one larger than the maximal clique size of $\Gamma$. Given a coset $P \in \mathcal{R}_0$, there is a collection of cosets $\mathcal{A} = \{ A_{\alpha} \}_{ \alpha \in \mathcal{I} } \in \mathcal{R}_0'$ such that 
	
	\begin{enumerate}
		\item $\phi(P) \subset \cup_{ \alpha \in \mathcal{I} } N_D(A_{\alpha})$ 
		\item Every $A \in \mathcal{A}$ is a member of  $Z' = \phi_0(P) \in \mathcal{R}_1'$.
		\item For each $A, B \in \mathcal{A}$, there is a sequence $A = A_1, A_2, ..., A_n = B$ such that $n \le 10$, $A_i \in \mathcal{A}$ for each $1 \le i \le n$ and $N_D(A_i) \cap A_{i+1}$ has infinite diameter for each $i <n$.
	\end{enumerate}
\end{lemma}

\begin{proof}
	Let $P = gW_{H} \in \mathcal{R}_0$ where $H = T_1 \star T_2 \in \Omega(\Gamma)$. Note that $W_{T_i}$ is infinite. By Lemma \ref{lemma_finding_flats}, choose a flat $F = f_1 \times f_2$ contained in $P$ where $f_i$ is a bi-infinite geodesic. Since $\Gamma$ is triangle-free, it follows that, with possibly reordering indices, $f_i \subset W_{T_i}$.
	
	We first show that $\phi_0(P)$ does not depend on the choice of a flat in $\Sigma_{\Gamma'}$. Let $F'$ and $F''$ be flats in $\Sigma_{\Gamma'}$, each Hausdorff close to $\phi(F)$. It follows that $F'$ is Hausdorff close to $F''$. By Lemma \ref{lemma_unique_membership}, there are cosets $P', P'' \in \mathcal{R}_0'$ which respectively contain $F'$ and $F''$. Hence, some finite neighborhood of $P'$ intersects $P''$ in an infinite diameter set. By Lemma \ref{lemma_infinite_intersection}(3), $P'$ and $P''$ are members of an unique coset $Z' \in \mathcal{R}_1'$ (uniqueness follows by Lemma \ref{lemma_unique_membership}).
	
	We next show that $\phi_0$ does not depend on the choice of flat $F \subset P$. Additionally, the work done in proving this sets up the proof for the other claims of the lemma. Let $P$ and $F$ be as before, and let $\hat{F} = g_1 \times g_2$ be another flat in $P$ with $g_i$ a bi-infinite geodesic in $W_{T_i}$.
	
	There is a bi-infinite geodesic $h_1$ in $W_{T_1}$ which intersects both $f_1$ and $g_1$ (one can readily check $W_{T_1}$ has the geodesic extension property). Similarly, there is a bi-infinite geodesic $h_2$ in $W_{T_2}$ which intersects both $f_2$ and $g_2$. Consider now the following series of flats: 
	\[F  = F_1 =  f_1 \times f_2, ~ F_2 = h_1 \times f_2, ~ F_3 = g_1 \times f_2 \]
	\[F_4 = g_1 \times h_2, ~ F_5 = g_1 \times g_2  = \hat{F} \]	
	$F_i \cap F_{i+1}$ has infinite diameter for $1 \le i < 5$. By Theorem \ref{thm_huang}, there is a constant $D>0$ depending only on $\phi$ so that $\phi(F_i) \subset N_D(F_i')$ for flats $F_1',...F_5'$ in $\Sigma_{\Gamma'}$. Furthermore, by Lemma \ref{lemma_flat_structures}, for each $i$, $F_i' \subset P_i'$ for some $P_i' \in \mathcal{R}_0'$. By Lemma \ref{lemma_infinite_intersection}(2), we may form the following sequence of cosets in $\mathcal{R}_0$:
	\[S_1 = P_1',~ S_2 = P_1'',~ S_3 = P_2',~ S_4 = P_2'',~...~,S_{7}= P_4',~ S_8 = P_4'',~ S_{9} = P_5' \]
	such that $N_D(S_j) \cap (S_{j+1})$ is infinite for each $j$, $1 \le j < 10$. Furthermore, by Lemma \ref{lemma_unique_membership} there is a unique $Z' \in \mathcal{R}_i'$ such that $S_j$ is a member of $Z'$ for each $j$, $1 \le j \le 10$. Hence, $\phi_0$ is well-defined.
	
	Given $P \in \mathcal{R}_0$, define $\mathcal{A}$ as follows:	
	\[\mathcal{A} = \{P' \in \mathcal{R}_0' ~ | ~ \text{there exists a flat, } F' \subset P', \text{Hausdorff close to } \phi(F) \text{ for some flat } F \subset P \}\]
	By Lemma \ref{lemma_finding_flats}, every $p \in P$ is contained in a flat, so claim 1 of the lemma follows.  By the work done above, $\mathcal{A}$ satisfies claims 2 and 3 of the lemma.
\end{proof}

We now extend $\phi_0$ to a new map, $\phi_1$. 

\begin{definition}
	Define the map
	\[ \phi_1: \mathcal{R}_1 \to \mathcal{R}_1' \]
	as follows. Given $Z \in \mathcal{R}_1$, choose $P \in \mathcal{R}_0$ such that $P$ is a member of $Z$. Set $\phi_1(Z) = \phi_0(P)$.
\end{definition}

\begin{lemma} \label{lemma_well_defined}
	$\phi_1$ is well defined.
\end{lemma}

\begin{proof}
	Let $Z \in \mathcal{R}_1$ and let $P, Q \in \mathcal{R}_0$ be members of $Z$. We will show that $\phi_0(P) = \phi_0(Q)$. 
		
	Since $P$ and $Q$ are both members of $Z$, by Lemma \ref{lemma_coarse_intersection_degree}, there is an finite sequence of elements in $\mathcal{R}_0$ starting with $P$ and ending with $Q$ so that the $M$ neighborhood of an element of this sequence has infinite diameter intersection with the next element of the sequence. Therefore, it is enough to show the claim for the case when $N_M(P) \cap Q$ has infinite diameter.
	
	Suppose $\phi_0(P) = Y'$ and $\phi_0(Q) = Z'$, for $Y', Z' \in \mathcal{R}_1'$. Let $\mathcal{A} = \{A_{\alpha}\}_{\alpha \in \mathcal{I}}$ and $\mathcal{B} = \{B_{\alpha}\}_{\alpha \in \mathcal{J}}$ be members respectively of $Y'$ and $Z'$, as in Lemma \ref{lemma_lambda_1_images}, so that $\phi(P) \subset \cup_{\alpha \in \mathcal{I}}N_D(A_{\alpha})$ and $\phi(Q) \subset \cup_{\alpha \in \mathcal{J}}N_D(B_{\alpha})$.  Let $A = \cup_{\alpha \in \mathcal{I}}{A_\alpha}$ and $B = \cup_{\alpha \in \mathcal{J}}{B_\alpha}$. Since $N_M(P) \cap Q$ has infinite diameter, there is a constant $C$ so that $A \cap N_{C}(B)$ has infinite diameter.
	
	Let $x, y$ be points in $A \cap N_{C}(B)$ distance $R$ apart from one another for $R$ large enough (to be later determined). Let $x'$ and $y'$ be points in $B$ distance at most $C$ from $x$ and $y$ respectively. Let $f_1$ be a geodesic from $x$ to $x'$ and $f_2$ a geodesic from $y$ to $y'$.
	
	As in Lemma \ref{lemma_lambda_1_images}, let $A_1,A_2...,A_n$ be a sequence of cosets in $\mathcal{A}$ such that $x \in A_1$, $y \in A_n$, $N_D(A_i) \cap A_{i+1}$ has infinite diameter and $n \le 10$. Let $a_1,...,a_{2n}$ be a sequence of geodesics such that $a_i \in A_i$ for $i$ odd and $|a_i| \le D$ for $i$ even, and such that the concatenation of geodesics $a = a_1 * ... * a_{2n}$ is a path from $x$ to $y$. Let $\hat{a} = \hat{a}_1 * \hat{a}_2 * ... * \hat{a}_{2n}$ be the geodesic obtained by deleting generators in the given expression of $a$ (this is possible by the deletion criteria of Coxeter groups, see for instance \cite[Theorem 1.5.1]{BB}). Define a sequence of elements in $\mathcal{B}$, $\{B_1, B_2,...,B_m \}$, and the geodesic $\hat{b} = \hat{b}_1*...*\hat{b}_{2m}$ similarly.
	
	Consider the disk diagram $\mathcal{D}$ with boundary path $f_1*\hat{b}*f_2^{-1}\hat{a}^{-1}$, and let $M'$ be a constant one larger than the maximal clique size of $\Gamma'$. Assuming $R$ was chosen sufficiently large it follows that $M'$ dual curves to $\hat{a}_j$ intersect $\hat{b}_k$, for some odd integers $j$ and $k$. Let $\mathcal{D}'$ be a minimal subdiagram of $\mathcal{D}$ which contains each of these dual curves. Let $p'$ denote the boundary path of $\mathcal{D}'$, and let $a'$ and $b'$ respectively be the subpaths of $\hat{a}_j$ and $\hat{b}_k$ which are contained in $p'$. 
	
	By applying \cite[Lemma 2.6]{Wise} twice, we can find another disk diagram $\mathcal{D}''$, with boundary path the same as that of $\mathcal{D}'$, except that $a'$ and $b'$ are replaced respectively with geodesics $a''$ and $b''$. Furthermore, these choices can be made so that no two dual curves to $a''$ intersect one another in $\mathcal{D}''$, no two dual curves to $b''$ intersect one another, and $M'$ curves dual to $a''$ still intersect $b''$. By Tit's solution to the word problem applied to right-angled Coxeter groups (see for instance \cite{Dav}, Chapter 3.4), the words $a''$ and $b''$ are just a reordering of the generators in the given expression for respectively $a'$ and $b'$. Hence the image of $a''$ and $b''$ in $\Sigma_{\Gamma'}$ are contained respectively in $A_j$ and $B_k$. 		  
	
	It follows the image of $\mathcal{D}''$ in $\Sigma_{\Gamma'}$ contains an Euclidean rectangle with a side of length $M'$ contained in $A_j$ and the opposite side contained in $B_k$. By the same argument as used in Lemma \ref{lemma_infinite_intersection}, it follows that $A_j$ and $B_k$ are members of the same coset of $\mathcal{R}_1'$. Hence, $Y' = Z'$, and this proves the claim. 
\end{proof}

\begin{lemma} \label{lemma_hausdorff_close}
	There is a constant $C$, only depending on the quasi-isometry constants of $\phi$, such that given $Z \in \mathcal{R}_1$, $\phi(Z)$ is $C$--Hausdorff close to $Z' = \phi_1(Z) \in \mathcal{R}_1'$. 
\end{lemma}
\begin{proof}
	 Let $M$ be one larger than the maximal clique size of $\Gamma$. Since $Z$ is contained in the $M$ neighborhood of its members in $\mathcal{R}_0$, by Lemma \ref{lemma_lambda_1_images} and Lemma \ref{lemma_well_defined}, $\phi(Z) \subset N_D(Z')$. 
	 
	 Let $P \in \mathcal{R}_0$ be a member of $Z$ and $F$ a flat in $P$. Let $F'$ be a flat $D$--Hausdorff close to $\phi(F)$ and $P' \in \mathcal{R}_0'$ be a member of $Z'$ containing $F'$. Let $\phi^{-1}$ be the quasi-isometric inverse of $\phi$. We define $\phi^{-1}_0$ and $\phi^{-1}_1$ similarly as to $\phi_0$ and $\phi_1$. It follows that $\phi^{-1}(F')$ is $C'$--Hausdorff close to $F$, for some $C'>0$. Hence, we have that $\phi^{-1}_0(P') = Z$ and $\phi^{-1}_1(Z') = Z$. Therefore, $\phi^{-1}(Z') \subset N_{D'}(Z)$ for some $D'>0$. This proves the claim. 
\end{proof}

\begin{lemma} \label{lemma_phi_1_bijection}
	$\phi_1$ is a bijection.
\end{lemma}
\begin{proof}
	To prove $\phi_1$ is injective, suppose $\phi_1(Z_1) = \phi_1(Z_2)$ for $Z_1, Z_2 \in \mathcal{R}_1$. By Lemma \ref{lemma_hausdorff_close}, $Z_1$ and $Z_2$ are Hausdorff close. By Lemma \ref{lemma_equal_cosets}, $Z_1 = Z_2$.
	
	To show $\phi_1$ is surjective, let $Z' \in \mathcal{R}_1'$ and $F'$ be a flat in $Z'$. $\phi^{-1}(F')$ is Hausdorff close to some flat $F$ contained in some $Z \in \mathcal{R}_1$. Therefore, $\phi_1(Z) = Z'$.
\end{proof}

We can now prove the main theorem of this section: 

\begin{proof}[Proof of Theorem \ref{thm_qi_invariant}]
	Suppose first the hypergraph index of $\Gamma$ is $\infty$. By Remark \ref{rmk_infinite_hyp_index}, $W_{\Gamma}$ is relatively hyperbolic. It follows by \cite[Theorem I]{BHSC} and \cite[Theorem 4.8]{BDM} that relative hyperbolicity is a quasi-isometry invariant in the setting of right-angled Coxeter groups. Therefore, $W_{\Gamma'}$ is relatively hyperbolic and its hypergraph index is $\infty$ as well.
	
	For the next case, suppose the hypergraph index of $\Gamma$ is $0$. Wide right-angled Coxeter groups (see sections \ref{sec_thick_overview} and \ref{sec_thick_bounds} for background) are precisely those with hypergraph index $0$  \cite[Proposition 2.11]{BHSC}. Wideness of a group is a quasi-isometry invariant, thus it follows $W_{\Gamma'}$ is also wide and has hypergraph index $0$ as well.
	
	For the general case, we may assume that the hypergraph index of $\Gamma$ is $h$, where $1 \le h < \infty$. By Lemma \ref{lemma_phi_1_bijection} we have a bijection $\phi_1: \mathcal{R}_1 \to \mathcal{R}'_1$ and by Lemma \ref{lemma_hausdorff_close}, $\phi$ sends any element $Z \in \mathcal{R}_1$ $C$--Hausdorff close to $\phi_1(Z) \in \mathcal{R}_1'$. By Remark \ref{rmk_coarse_intersection_qi}, the coarse intersection degree of $\mathcal{R}_1$ is the same as that of $\mathcal{R}_1'$. By Lemma \ref{lemma_coarse_intersection_degree}, $\Gamma$ and $\Gamma'$ have the same hypergraph index. 
\end{proof}

\section{Thickness bounds} \label{sec_thick_bounds}

We show the hypergraph index of the right-angled Coxeter group yields upper bounds on the group's order of thickness, order of algebraic thickness and divergence. Throughout, $\Gamma$ will always denote a simplicial graph, and $W_{\Gamma}$ is the associated right-angled Coxeter group. We emphasize that for the remainder of the paper, no restrictions are placed on $\Gamma$.

The next two theorems summarize results in the literature that show there are many equivalent ways of describing thick of order 0 and thick of order 1 right-angled Coxeter groups. The proof follows from work in \cite{BHSC}, \cite{BFHS}, \cite{DT} and \cite{Lev}. For a definition of a CFS graph, see \cite{BFHS}. For the definition of divergence used, see \cite{Lev}.

\begin{theorem}[Thick of order 0 classification] \label{thm_thick_0_classification}
	The following are equivalent:
	\begin{enumerate}
		\item $\Gamma = A \star B$, with $A$ and $B$ each containing a pair of non-adjacent vertices
		\item $W_{\Gamma}$ is algebraically thick of order 0
		\item $W_{\Gamma}$ is thick of order 0
		\item The divergence of $W_{\Gamma}$ is linear
		\item $\Gamma$ has hypergraph index 0
	\end{enumerate}
\end{theorem}
\begin{proof}
	The implications $1 \to 2 \to 3 \to 4 \to 1$ are either obvious or follow from \cite{BHSC}. $5 \to 1 \to 5$ follows from the definition of hypergraph index.
\end{proof}

\begin{theorem}[Thick of order 1 classification] \label{thm_thick_1_classification}
The following are equivalent:
\begin{enumerate}
\item $\Gamma$ is CFS and $\Gamma \neq A \star B$, with $A$ and $B$ each containing a pair of non-adjacent vertices
\item $W_{\Gamma}$ is algebraically thick of order 1
\item $W_{\Gamma}$ is thick of order 1
\item The divergence of $W_{\Gamma}$ is quadratic
\item $\Gamma$ has hypergraph index 1
\end{enumerate}
\end{theorem}
\begin{proof}
The implication $1 \to 2 \to 3 \to 4$ are either obvious or follow from \cite{BHSC} and \cite{BFHS}. $4 \to 1$ follows from \cite{Lev}. $5 \to 3$ follows from Theorem \ref{thm_thick_upper_bound} below. $1 \to 5$ is an easy exercise. 
\end{proof}

The hypergraph index yields an upper bound for the order of thickness:

\begin{theorem} \label{thm_thick_upper_bound}
If $W_{\Gamma}$ has hypergraph index $h$, then $W_{\Gamma}$ is thick of order at most $h$.
\end{theorem}
\begin{proof}
Subspaces in $\mathcal{R}_0$ are uniformly wide as they consist of finitely many isometry classes of products of infinite groups. For any $i$, the elements of $\mathcal{R}_i$ are convex by Lemma \ref{lemma_special_subgroups}. Therefore, Lemma \ref{lemma_coarse_intersection_degree} shows that $\mathcal{R}_i$ is a tight network of subspaces with respect to $\mathcal{R}_{i-1}$. Finally, $\mathcal{R}_h$ consists of all of $W_{\Gamma}$. It follows that $W_{\Gamma}$ is thick of degree at most $h$.
\end{proof}

The next corollary follows from the above theorem and \cite[Corollary 4.17]{BD}.

\begin{corollary}
If $W_{\Gamma}$ has hypergraph index $h$, then the divergence of $W_{\Gamma}$ is bound above by a polynomial of degree $h+1$.
\end{corollary}

The hypergraph index also provides an upper bound on the order of algebraic thickness:

\begin{theorem} \label{thm_alg_thick_upper_bound}
If $W_{\Gamma}$ has hypergraph index $h>0$, then $W_{\Gamma}$ is algebraically thick of order at most $2h-1$.
\end{theorem}
\begin{proof}
The proof will be by induction. The base case when $h = 1$ follows from Theorem \ref{thm_thick_1_classification}.

Assume the claim is true for graphs of hypergraph index $h$ and suppose $\Gamma$ has hypergraph index $h+1$. Let $\{E_1,...,E_m \}$ be hyperedges of $\Lambda_h(\Gamma) = \Lambda_h$ which are not strip subgraphs ($\Lambda_h$ is the hypergraph from Definition \ref{def_lambda_hypergraphs}). By the induction hypothesis, the subgroups $\{W_{E_1},...,W_{E_m} \}$ are algebraically thick of order at most $2h-1$. Let $\{S_1,...,S_r\}$ be hyperedges of $\Lambda_h$ corresponding to strip subgraphs. 

Since $\Lambda$ has hypergraph index $h+1$ and by Remark \ref{rmk_strip_subgroup_intersection}, for each $S_i$ there is some $E_j \in \{E_1,...,E_m \}$ such that $S_i \cap E_j$ contain two non-adjacent vertices. Set $\bar{S}_i = E_j$. By \cite[Proposition A.2]{BHSC}, it follows that $W_{S_i \cup \bar{S}_i}$ is thick of order at most $(2h-1) +1 = 2h$. $W_{\Gamma}$ is then algebraically thick of order at most $2h+1$ is respect to the special subgroups: 
\[ \{W_{E_1},...,W_{E_m}\} \cup \{W_{S_1 \cup \bar{S}_1},..., W_{S_r \cup \bar{S}_r} \} \] 
\end{proof}

\section{Thickness $\neq$ algebraic thickness} \label{sec_thick_not_alg}

As described in the introduction, there are known examples of groups which are thick of order $1$, but are not algebraically thick of order $1$. However, by Theorem \ref{thm_thick_1_classification} we know for the class of right-angled Coxeter groups thickness of order $1$ is equivalent to algebraic thickness of order $1$. The following natural question suggests itself: for at least the class of right-angled Coxeter groups, is algebraic thickness of order $n$ equivalent to thickness of order $n$?

The goal of this section is to provide a negative answer to the above question: for each positive integer $n >1$, there exists a right-angled Coxeter group that is thick of order $n$ but is not algebraically thick of order $n$. This is the main content of Theorem \ref{thm_counterexample} stated below. The proof of this theorem relies on some preliminary lemmas which we first prove. We begin by recalling the definition from \cite{Lev} of a rank $n$ pair.

\begin{definition}[Rank $n$ pair] \label{def_rank_n_pair}
Given distinct vertices $s, t \in \Gamma$, $(s, t)$ is a \textit{non-commuting pair} if $s$ is not adjacent to $t$ in $\Gamma$.  A non-commuting pair $(s, t)$ is rank 1 if $s, t$ are not contained in some induced square of $\Gamma$. Additionally $(s, t)$ are rank $n$ if either every non-commuting pair $(s_1, s_2)$, with $s_1, s_2 \in Link(s)$, is rank $n-1$ or every non-commuting pair $(t_1, t_2)$, with $t_1, t_2 \in Link(t)$, is rank $n-1$. 
\end{definition}

\begin{theorem} \label{thm_counterexample}
Given an integer $n>1$, let $\Gamma$ be a graph satisfying the following hypotheses:

\begin{enumerate}
\item There is a subgraph, $B \subset \Gamma$ such that $V(\Gamma) \setminus V(B)$ is just two vertices: $u$ and $v$.
\item $B$ has hypergraph index $n-1$
\item $Link(u)$ is two non-adjacent vertices of $B$. Similarly, $Link(v)$ is two non-adjacent vertices of $B$.
\item For all $s \in \Gamma - Star(u)$, $(u, s)$ is a rank $n$ pair. Similarly, for all $s \in \Gamma - Star(v)$, $(v, s)$ is a rank $n$ pair.
\end{enumerate}

It follows the divergence of $W_{\Gamma}$ is a polynomial of degree $n+1$, $W_{\Gamma}$ is thick of order $n$, and $W_{\Gamma}$ is algebraically thick of order $d$, where $n < d \le 2n-1$. Furthermore, for every $n > 1$ such a graph, $\Gamma$, exists.
\end{theorem}

Figure \ref{fig_counterexample} gives a family of graphs which can be readily checked to satisfy the hypotheses of Theorem \ref{thm_counterexample}. This family of graphs proves the last statement of the theorem, namely the existence of such graphs. 

\textbf{For the remainder of this section we fix an integer $n>1$ and a graph $\Gamma$ satisfying the hypotheses of Theorem \ref{thm_counterexample}. Furthermore we fix $u, v \in V(\Gamma)$ as in the statement of the theorem.}
 
\begin{figure}[htp]
	\centering
	
	\begin{overpic}[scale=.8]{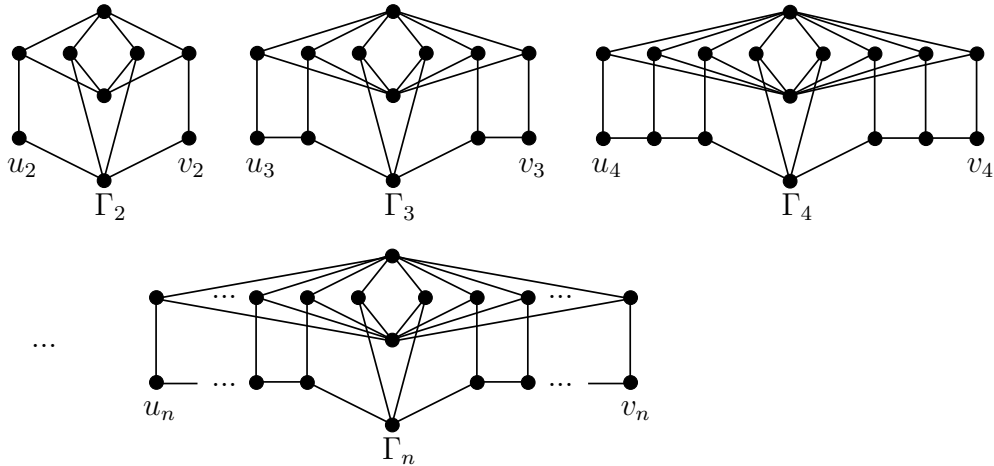}
		\put(8.5,22.5){$\Gamma_2$}
		\put(38.2,22.5){$\Gamma_3$}
		\put(79,22.5){$\Gamma_4$}
		\put(38,-2.5){$\Gamma_n$}
		
		\put(-.5,27){$u_2$}
		\put(17,27){$v_2$}
		
		\put(24,27){$u_3$}
		\put(52,27){$v_3$}
		
		\put(59.5,27){$u_4$}
		\put(98,27){$v_4$}
		
		\put(13.5,2){$u_n$}
		\put(62.5,2){$v_n$}
		
		\put(2,9){$...$}
		
		\put(20.5,5){$...$}
		\put(55,5){$...$}
		\put(20.5,14){$...$}
		\put(55,14){$...$}
	\end{overpic}
	
	\caption{The given family of graphs provides a family of right-angled Coxeter groups, $W_{\Gamma_n}$, for $n>1$. $W_{\Gamma_n}$ is thick of order $n$ but is algebraically thick of order strictly larger than $n$.} \label{fig_counterexample}
\end{figure}

$W_{\Gamma}$ decomposes as the amalgamated product $W_{\Gamma} = W_{star(u)} *_{link(u)} W_{B} *_{link(v)} W_{star(v)}$. It follows from Bass-Serre theory that $W_{\Gamma}$ acts on a tree $\mathcal{T}$, with fundamental domain the graph of groups shown in figure \ref{fig_graph_of_groups}. Fix this tree $\mathcal{T}$.
\vspace{.5cm}

\begin{figure}[htp]
\centering

\begin{overpic}[scale=1.5]{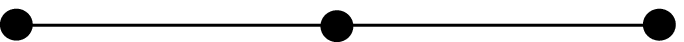}
\put(-15,-4){$W_{Star(u)} = \mathbb{Z}_2 \times (\mathbb{Z}_2 * \mathbb{Z}_2)$}
\put(10,4){$W_{Link(u)} = \mathbb{Z}_2 * \mathbb{Z}_2$}
\put(80,-4){$W_{Star(v)} = \mathbb{Z}_2 \times (\mathbb{Z}_2 * \mathbb{Z}_2)$}
\put(57,4){$W_{Link(v)} = \mathbb{Z}_2 * \mathbb{Z}_2$}
\put(47.5,-4){$W_B$}
\end{overpic}

\caption{Graph of groups corresponding to $W_{\Gamma}$.} \label{fig_graph_of_groups}
\end{figure}

Given $w$ a minimal length expression of a word in $W_{\Gamma}$, let $L_s(w)$ be the number of occurrences of the generator $s$ in $w$.

\begin{lemma} \label{lemma_hyp_isom_divergence}
Let $w \in W_{\Gamma}$ be a hyperbolic isometry of the Bass-Serre tree $\mathcal{T}$. The bi-infinite geodesic $...www...$, in the Davis complex $\Sigma_{\Gamma}$, has polynomial divergence of degree $n+1$.
\end{lemma}
\begin{proof}
Since $w$ is hyperbolic, by putting a reduced expression of $w^i$ into normal form, we see that either $L_{u}(w^i)$ or $L_{v}(w^i)$ grows linearly with $i$. Without loss of generality, assume $L_{u}(w^i)$ grows linearly with $i$. Given a reduced expression, $g$, of $w^i$, it follows that given two occurrences of $u$ in $g$ there must exist some $s \in \Gamma$, which is not adjacent to $u$, between these occurrences (i.e. $g = ...u...s...u...$). Since for any such $s$, $(s,u)$ forms a rank $n$ pair by a hypothesis of Theorem \ref{thm_counterexample}, by a slight modification of the proof of \cite[Theorem 7.9]{Lev}, it follows that the bi-infinite geodesic $...www...$ has polynomial divergence of degree $n+1$. 
\end{proof}

\begin{lemma} \label{lemma_thick_n_conj}
Any quasi-isometrically embedded thick of order $n-1$ subgroup is contained in a conjugate of $W_B$.
\end{lemma}
\begin{proof}
Let $G$ be such a thick of order $n-1$ quasi-isometrically embedded subgroup of $W_{\Gamma}$. Given $w \in G$, $w$ cannot act as a hyperbolic isometry of the Bass-Serre tree $\mathcal{T}$, for then by Lemma \ref{lemma_hyp_isom_divergence}, $G$ would have divergence at least a polynomial of degree $n+1$ which is not possible since thick of order $n-1$ groups have divergence at most $n$ by \cite[Corollary 4.17]{BD}. It follows that any $w \in G$ acts elliptically on $\mathcal{T}$. 

Since two elliptic isometries with disjoint fixed point set generate a hyperbolic element (see \cite[1.5]{CM}), we have that every element of $G$ is contained in some conjugate of $W_B$, some conjugate of $W_{Star(u)}$ or some conjugate of $W_{Star(v)}$. However, $W_{Star(u)}$ and $W_{Star(v)}$ are both virtually $\mathbb{Z}$ and so cannot contain a thick of order $n-1$ subgroup. Thus, $G$ must be contained in some conjugate of $W_B$.
\end{proof}

\begin{lemma} \label{lemma_infinite_index}
Let $\{G_1, ..., G_m \}$ be a finite set of subgroups contained in a conjugate of $W_B$. The subgroup, $G$, generated by $\cup_{i = 1}^n G_i$ is infinite index in $W_{\Gamma}$.
\end{lemma}

\begin{proof}

Let $H = \mathbb{Z}_2 * \mathbb{Z}_2$ and $a, b$ the canonical generators of $H$. Define the homomorphism $\phi: W_{\Gamma} \to H$ by the map on generators: $\phi(u) = a$, $\phi(v) = b$, and $\phi(s) = 1$ for $s \neq u, v$.

Let $w \in G$. We can write $w = g_1w_1g_1^{-1} g_2w_2g_2^{-1} ... g_kw_kg_k^{-1}$ where $w_i \in W_B$ for each $1 \le i \le k$. It  follows that $\phi(w) = 1$. 

For a contradiction, suppose $G$ is finite index in $W_{\Gamma}$. It follows for some $i>0$ large enough, $G$ must contain a word $w$ of one of the following forms: $w = (uv)^i$, $w = (vu)^i$, $w = (uv)^iu$ or $w = (vu)^iv$,
However, for each of these cases, $\phi(w) \neq 1$, a contradiction. 
\end{proof}

We are now in a position to prove Theorem \ref{thm_counterexample}: 

\begin{proof}[Proof of Theorem \ref{thm_counterexample}]
Given an integer $n>1$, a graph $\Gamma$ satisfying the hypotheses of the theorem exists by the family of examples given in Figure \ref{fig_counterexample}. In fact, one can construct many such families. Fix such a graph $\Gamma$.

It is immediate $\Gamma$ has hypergraph index $n$, as $B$ has hypergraph index $n-1$ and $\Gamma$ consists of the addition of two strip subgraphs to $B$. By Theorem \ref{thm_thick_upper_bound}, $W_{\Gamma}$ is thick of order at most $n$. By Lemma \ref{lemma_hyp_isom_divergence}, $W_{\Gamma}$ has divergence a polynomial of degree $n+1$. By the lower bound on thickness provided by the divergence function, $W_{\Gamma}$ is thick of order exactly $n$. 

By Lemma \ref{lemma_thick_n_conj} and Lemma \ref{lemma_infinite_index}, $W_{\Gamma}$ cannot be algebraically thick of order $n$ since no finite set of thick of order at most $n-1$ subgroups generate a finite index subgroup of $W_{\Gamma}$. By Theorem \ref{thm_alg_thick_upper_bound}, $W_{\Gamma}$ is algebraically thick of order $d$ where $n < d \le 2n -1$.
\end{proof}

\bibliographystyle{amsalpha}
\bibliography{mybibliography}
\end{document}